\documentclass[runningheads]{llncs}
\usepackage{graphicx}
\graphicspath{ {./images/} }
\usepackage[english]{babel}
\usepackage{algorithm}
\usepackage{enumerate}
\usepackage{appendix}
\usepackage{mathtools}
\usepackage{enumitem}
\usepackage{empheq}
\usepackage{tikz}
\usepackage{algpseudocode}
\usepackage{hyperref}
\hypersetup{
    colorlinks=true,
    linkcolor=blue,
    filecolor=magenta,      
    urlcolor=cyan
    }
\usepackage{amsfonts}
\usepackage{longtable}
\usepackage{epsfig}

\usepackage{amssymb}
\usepackage{amsmath}
\usepackage{import}
\usepackage{xifthen}
\usepackage{pdfpages}
\usepackage{transparent}
\usepackage{amsthm}
\makeatletter
\newcommand{\mylabel}[2]{#2\def\@currentlabel{#2}\label{#1}}
\makeatother

\makeatletter
\def\moverlay{\mathpalette\mov@rlay}
\def\mov@rlay#1#2{\leavevmode\vtop{%
   \baselineskip\z@skip \lineskiplimit-\maxdimen
   \ialign{\hfil$\m@th#1##$\hfil\cr#2\crcr}}}
\newcommand{\charfusion}[3][\mathord]{
    #1{\ifx#1\mathop\vphantom{#2}\fi
        \mathpalette\mov@rlay{#2\cr#3}
      }
    \ifx#1\mathop\expandafter\displaylimits\fi}
\makeatother

\newcommand{\cupdot}{\charfusion[\mathbin]{\cup}{\cdot}}

\usetikzlibrary{shapes,arrows}
\theoremstyle{remark}

\newtheorem{lem}[theorem]{Lemma}

\begin{document}
\title{The Median of Sierpi\'{n}ski Triangle Graphs}
%
%
\author{Kannan Balakrishnan\inst{1}\orcidID{0000-0003-1285-5949} \and
Manoj Changat\inst{2}\orcidID{0000-0001-7257-6031} \and
M V. Dhanyamol\inst{3}\orcidID{0000-0003-4487-7704} \and \\
Andreas M. Hinz\inst{4} \and
Hrishik Koley \inst{5} \and
Divya Sindhu Lekha \inst{3}\orcidID{0000-0002-6617-5378}
}
\authorrunning{K. Balakrishnan et al.}
%
\institute{Department of Computer Applications, Cochin University of Science and Technology, India 
\email{mullayilkannan@gmail.com}\\
\and
Department of Futures Studies, University of Kerala, Thiruvananthapuram, India
\email{mchangat@gmail.com}\\
\and
Indian Institute of Information Technology Kottayam, India
\email{\{dhanya,divyaslekha\}@iiitkottayam.ac.in}\\
\and
Mathematical Institute, LMU M\"{u}nchen, Munich, Germany
\email{hinz@math.lmu.de}\\
\and
Indian Statistical Institute, Bengaluru, India
\email{hrishik.math@outlook.com}
}
\maketitle     
\vspace{-1em}
\begin{abstract}
  The median $M$ of a graph $G$ is the set of vertices with a minimum total distance to all other vertices in the graph. In this paper, we determine the median of Sierpi\'{n}ski triangle graphs. Sierpi\'{n}ski triangle graphs, also known as Sierpi\'{n}ski gasket graphs of order $n$ are graphs formed by contracting all non-clique edges from the Sierpi\'{n}ski graphs of order ($n+1$).

\end{abstract}

\section{Introduction}
We deal with many networks in our day-to-day lives. The metric properties of graphs allow us to study the topological features of networks such as social networks, computer networks, and transportation networks in detail~\cite{hemi-2011}. These metric properties help identify important nodes, detect communities, and understand the flow of resources in the network. For example, in social networks, they are used to identify influential individuals or tightly-knit communities.

As a metric property, the median of a graph is significant and finding it makes designing networks much easier~\cite{babe-1981}. For example, if we can successfully identify the median of a graph in a transportation network then it helps us optimize the placement of various facilities, in such a way, that it minimizes the overall distances as well as makes any facility easily accessible on the network. Similarly, in the case of the spread of diseases, finding medians aids in optimizing the placement of hospitals and clinics for effective resource allocation or deciding vaccination strategies.

The median vertices of a graph are the vertices with the minimum total distance from all other vertices in it. Given a connected (simple) graph $G=(V;E)$ with $|G| \ge 2$, the distance between two vertices is given by the length of a shortest path between them. The {\bf total distance} of a vertex $v \in V$ is defined as the sum of distances from $v$ to all other vertices. Thus, {\bf median vertices} are vertices that minimise this total distance and form the {\bf median} of $G$.

Let $G=(V, E)$ be a connected (simple) graph with $|G|\geq 2$, and with the canonical distance function $d$. The {\bf total distance} $d(v)$ of a vertex $v$ is defined as the sum of distances $d(v,u)$ to all other vertices $u$. A {\bf median vertex} is a vertex that minimizes the total distance over $V$. The {\bf eccentricity} of $v\in V$ is $\epsilon(v)=\max_{u\in V}d(v,u)$ and its {\bf average distance} is $\overline{d}(v)=d(v)/(|G|-1)$. We then have, 
\begin{lem}\label{lem:avdisteccentric}
	$\text{For all }v\in V:\;|G|-1\leq d(v)\leq (|G|-1)\epsilon(v),\;1\leq \overline{d}(v)\leq \epsilon(v)\,.$
\end{lem}
\noindent This is evident, with equalities for complete graphs.\\

The \textbf{proximity} of G is $$\mathrm{prox}(G) = \mathrm{min} \{\overline{d}(v) | v \in V \}$$ and the \textbf{remoteness} of G is $$\mathrm{rem}(G) =\mathrm{max}\{\overline{d}(v) | v \in V \}.$$

The {\bf median} is the set of median vertices
$$
M(G)=\{v\in V\ |\ \overline{d}(v)=\mathrm{prox}(G)\}\,.
$$

Sierpi\'{n}ski triangle graphs are a very common fractal structure that is widely studied and can be found in nature~\cite{stew-1995}. Recently, Hinz et al. ~\cite{hinz-2022} have done a detailed study on the metric properties like center and periphery of Sierpi\'{n}ski triangle graphs. In the present paper, we study the median of Sierpi\'{n}ski triangle graphs. 

\section{Median of Sierpi\'{n}ski Graphs}

{\em Switching Tower of Hanoi} (STH) is a variation of the {\em Tower of Hanoi} (ToH) game~\cite{hinz-2018}. The goal of STH is to shift the tower of discs from one peg to another. Let us label the discs $1$ through $n$, based on their increasing size. The set of rules in ToH apply in STH only for the smallest disc; otherwise we can switch the entire tower of discs~1 through $d$ if they lie in order on a single peg, with the disc~$d+1$, if it lies on top of another peg. Therefore, there are moves of type~0 (only disc~1 is moving) and of type~1 (a non-empty subtower and a disc are switched). 

We model ToH and STH with $n\in\mathbb{N}_0$ discs by {\bf Hanoi graphs} $H^n$ and {\bf Sierpi\'{n}ski graphs} $S^n$, respectively. Any legal distribution of the discs among the pegs, labelled by elements of $T:=\{0,1,2\}$ for convenience, represents a {\bf state} $s=s_n\ldots s_1\in T^n$, where $s_d$ is the peg disc $d\in [n]:=\{1,\ldots,n\}$ is positioned. Hence $V(H^n)=T^n=V(S^n)$ and
$$
E(H^n)=\left\{\{\underline{s}ik^{d-1},\underline{s}jk^{d-1}\}\ |\ \underline{s}\in T^{n-d},\,\{i,j,k\}=T,\,d\in [n]\right\}\,,$$
$$
E(S^n)=\left\{\{\underline{s}ij^{d-1},\underline{s}ji^{d-1}\}\ |\ \underline{s}\in T^{n-d},\,\{i,j\}\in\textstyle{{T\choose 2}},\,d\in [n]\right\}=E_0^n \cupdot E_1^n\,,
$$
where $E_0^n$ and $E_1^n$ represent the moves to type 0 ($d=1$) and 1 ($d>1)$, respectively. Moreover, we also define the \textbf{extreme vertices} for both $H^n$ and $S^n$ as the vertices corresponding to the states of the ToH, in which all of the $n$ discs lie on a single peg. Note that $H^n\cong S^n$; see (\cite{hinz-2018}, Figure~4.4).

From \cite{med.SG} we know that
$
M(S^0)=\{{\rm e}\}$ with the empty word ${\rm e}$, $M(S^1)=T,\,M(S^2)=\left\{ij\ |\ \{i,j\}\in\textstyle{{T\choose 2}}\right\},
$ and $M(S^n)=\{ijk^{n-2}\ |\ i,j,k\in T,\,j\neq i\neq k\}$ for $n\geq 3$, such that
$$
|M(S^0)|=1,\,|M(S^1)|=3,\,|M(S^2)|=6,\,{\rm and}\,|M(S^n)|=12\;{\rm for}\,n\geq 3\,.
$$

For example, we may consider $S^3_3$, and we can see and deduce from Figure \ref{fig:S33}, that $M(S^3_3) = \{011,012,021,022,200,201,210,211,122,120,102,100\}$.

\tikzset{every picture/.style={line width=0.75pt}} 
\begin{figure}
    \centering
\begin{tikzpicture}[x=1pt,y=1pt,yscale=-1,xscale=1]

\draw    (331,80.41) -- (405,209.41) ;
\node at (414,209.41) {$222$}; 
\draw    (331,80.41) -- (257,209.41) ;
\node at (331,73.41) {$000$}; 
\draw    (405,209.41) -- (257,209.41) ;
\node at (248,209.41) {$111$};
\draw    (337,209.41) -- (372.24,151.5) ;
\draw    (325,209.41) -- (289.76,151.5) ;
\draw    (295,142) -- (367,142) ;
\node at (340,216.41) {$211$};
\node at (322,216.41) {$122$};
\node at (381.24,151.5) {$200$};
\node at (280.76,151.5) {$100$};
\node at (286,140) {$011$};
\node at (376,140) {$022$};
\draw    (326,142) -- (310,117) ;
\draw    (336,142) -- (352,117) ;
\draw    (314,109) -- (348,109) ;
\node at (323,148) {$012$};
\node at (301,117) {$010$};
\node at (339,148) {$021$};
\node at (361,117) {$020$};
\node at (305,109) {$001$};
\node at (357,109) {$002$};
\draw    (376,209.41) -- (390.5,185.41) ;
\draw    (366,209.41) -- (351.5,185.41) ;
\draw    (356,177.41) -- (387,177.41) ;
\node at (379,216.41) {$212$};
\node at (400,185.41) {$220$};
\node at (363,216.41) {$221$};
\node at (342,185.41) {$210$};
\node at (347,177.41) {$201$};
\node at (396,177.41) {$202$};
\draw    (296,209.41) -- (310.5,185.41) ;
\draw    (286,209.41) -- (270.5,185.41) ;
\draw    (275,177.41) -- (306,177.41) ;
\node at (299,216.41) {$121$};
\node at (319.5,185.41) {$102$};
\node at (283,216.41) {$112$};
\node at (261.5,185.41) {$101$};
\node at (266,177.41) {$110$};
\node at (315,177.41) {$120$};
\end{tikzpicture}
\caption{Sierpi\'{n}ski graph $S^3_3$}
\label{fig:S33}
\end{figure}
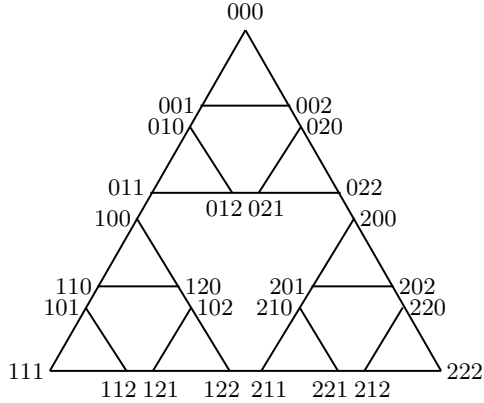

\section{Small Sierpi\'{n}ski Triangle Graphs}

Sierpi\'{n}ski triangle graphs are a variation of the Sierpi\'{n}ski graphs. Before delving into how Sierpi\'{n}ski graphs are modified to get Sierpi\'{n}ski triangle graphs, we define \textbf{primitive vertices} of Sierpi\'{n}ski triangle graphs. We define the \textbf{primitive vertices} of $\widehat{S}^n$ as the extreme vertices of $\widehat{S}^n$, and they are denoted by $\widehat{i}$, where $i \in T$.

In a shortest path between two vertices of $S^n$, $n\geq 2$, a non-terminal move of type~0 is always followed immediately by a move of type~1 and vice versa. For the latter there is only one choice, It can therefore be considered as a {\em forced move} and left out in an {\em accelerated count} of moves. We obtain the {\bf Sierpi\'{n}ski triangle graph} $\widehat{S}^n$, $n\in\mathbb{N}_0$, from Sierpi\'{n}ski graph $S^{n+1}$ by contracting all edges from $E_1^{n+1}$. The end vertices 
$s_1:=\underline{s}ij^{d-1}$ and $s_2:=\underline{s}ji^{d-1}$ of such an edge in $S_p^{n+1}$ with $1<d\leq n+1$ and $\underline{s}=s_{n+1}\ldots s_{d+1}\in T^{n-d+1}$ and $\{i,j\}\in{T\choose 2}$
form a combined state $\underline{s}k$, $k=3-i-j$, i.e.~vertex in $\widehat{S}^n$. The {\bf extreme vertices} $i^{n+1}$, $i\in T$, of $S^{n+1}$ are not incident with a move of type~1, so they carry over to {\bf primitive vertices} $\widehat{i}$ of $\widehat{S}^n$, forming the set $\widehat{T}\subset V(\widehat{S}^n)$. The moves of type~0 in $S^{n+1}$ forms the edge set in $\widehat{S}^n$. In other words, only moves of type~0 in $S^{n+1}$ are counted in $\widehat{S}^n$. More formally,
$$
V(\widehat{S}^n)=\widehat{T}\cup \bigcup_{d=2}^{n+1} T^{n-d+2}
$$
and $V(\widehat{S}^n)\setminus \widehat{T}\cong E_1^{n+1}$, $E(\widehat{S}^n)\cong E_0^{n+1}$, where $\{s,t\}\in E(\widehat{S}^n)$ iff $\{s_a,t_b\}\in E_0^{n+1}$ for some $a,b\in \{1,2\}$.

Clearly, $\widehat{S}^0_3 \cong S^1_3$. Therefore, $|M(\widehat{S}^0_3)|=|M(S^1_3)|=|V(S^1_3)|=|V(\widehat{S}^0_3)|=3$. By definition of proximity and median, it is trivial that $M(\widehat{S}^0_3)= V(\widehat{S}^0_3)=\widehat{T}=\{\widehat{0}, \widehat{1}, \widehat{1}\}$. (Refer to Figure ~\ref{fig:st}.)

Obviously, $M(\widehat{S}^1_3)= \{0,1,2\}$ which comprises the vertices lying on the inner ring, as shown in Figure~\ref{fig:st}. So $|M(\widehat{S}^0_3)|=3$.

\begin{figure}[htbp]
\centering

	\begin{tikzpicture}[scale=0.5]
	\draw (0,0) --(3.5,7);
	\draw (3.5,7) --(7,0);
	\draw (7,0) --(0,0);
	
	\node at (3.5,7.5) {$\widehat{0}$};
	\node at (-0.3,-0.2) {$\widehat{1}$};
	\node at (7.4,-0.2) {$\widehat{2}$};	
	
	\filldraw [fill=red] (0,0) circle (2pt);
        \filldraw [fill=red] (3.5,7) circle (2pt);
        \filldraw [fill=red] (7,0) circle (2pt);
	
	
	\draw (10,0) --(13.5,7);
	\draw (13.5,7) --(17,0);
	\draw (17,0) --(10,0);
	
	\node at (13.5,7.4) {$\widehat{0}$};
	\node at (9.8,-0.2) {$\widehat{1}$};
	\node at (17.2,-0.2) {$\widehat{2}$};
	
	\draw (11.75,3.5) --(13.5,0);
	\draw (15.25,3.5) --(13.5,0);
	\draw (11.75,3.5) --(15.25,3.5);
	
	\node at (11.4,3.5) {$2$};
	\node at (15.6,3.5) {$1$};	
	\node at (13.5,-0.4) {$0$};

        \filldraw [fill=red] (11.75,3.5) circle (2pt);
        \filldraw [fill=red] (15.25,3.5) circle (2pt);
        \filldraw [fill=red] (13.5,0) circle (2pt);
        
\end{tikzpicture}

\caption{Sierpi\'{n}ski triangle graphs $\widehat{S}^0_3$ and $\widehat{S}^1_3$, respectively}
\label{fig:st}
\end{figure}
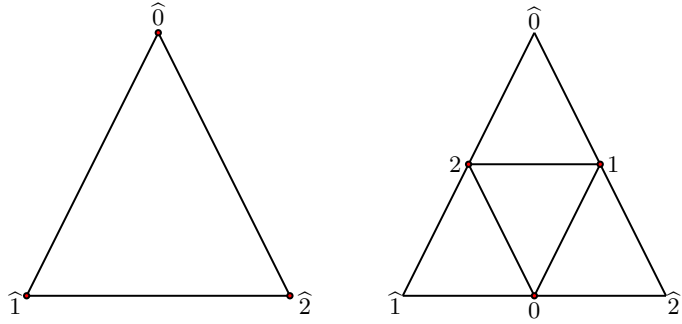

\section{Median of Sierpi\'{n}ski Triangle Graphs}
\subsubsection{Definitions:}
\begin{enumerate}
    \item $\widehat{d}_n(s,t)$ is defined as the distance between any two vertices $s$ and $t$ in the Sierpi\'{n}ski triangle graph $\widehat{S}^n_3$.
    \item $d_n(s,t)$ is defined as the distance between any two vertices $s$ and $t$ in the Sierpi\'{n}ski graph $S^n_3$.
    \item $\widehat{d}_n(s)$ is defined as the total distance of any vertex $s$ from all other vertices in the Sierpi\'{n}ski triangle graph $\widehat{S}^n_3$.
    \item $d_n(s)$ is defined as the total distance of any vertex $s$ from all other vertices in the Sierpi\'{n}ski graph $S^n_3$.
    \item $\widehat{d'}_n(s)$ is defined as the total distance of any vertex $s$ from all other vertices except the primitive vertices in the Sierpi\'{n}ski triangle graph $\widehat{S}^n_3$.
    \item $d'_n(s)$ is defined as the total distance of any vertex $s$ from all other vertices except the extreme vertices and the vertex that forms a non-clique edge with $s$ in the Sierpi\'{n}ski graph $S^n_3$.
\end{enumerate}
Our main goal is to find the median of the Sierpi\'{n}ski triangle graphs.

\begin{quote}
      \begin{theorem} \label{hyp:basic}
    $|M(\widehat{S}^n_3)|=6$ for all $ n \ge 2$.

    \
    
    The median vertices of $\widehat{S}^n_3$ are the vertices formed by contracting edges in $S^{n+1}_3$, whose end vertices are the median of $S^{n+1}_3$.
    
    \
    
    M($\widehat{S}^n_3$)=$\widehat{M}^n_0$:=$\{0,1,2,00,11,22\}$ and $|M(\widehat{S}^n_3)|=6$ for $n \in \mathbb{N}_0$ and \\ $n \geq 2$.
\end{theorem}

For the proof we need some preliminaries.
\end{quote}

\begin{theorem}\label{thm:primitive}(Equation 6 in Section 3.1 of ~\cite{hinz-2022})
    For any given vertex $s \in V(\widehat{S}^n_3)$, the sum of distances of any given vertex from the primitive vertices is
        \begin{align*}\widehat{d}_n(s)-\widehat{d'}_n(s)=2^{n+1}\end{align*}
\end{theorem}
\begin{theorem}\label{thm:extreme}(Corollary 4.7 in ~\cite{hinz-2018})
    For any given vertex $s \in V(S^n_3)$, the sum of distances of any given vertex from the extreme vertices is    \begin{align*}d_n(s)-d'_n(s)=2(2^n-1)\end{align*}
\end{theorem}
\begin{theorem}\label{thm:2vert}
    For any two given vertices $s,t \in V(\widehat{S}^n_3)$,    
    \begin{equation}\label{ref:eq1}
    \widehat{d}_n(s,t)=\left\lceil \frac{d_{n+1}(s_1,t_1)+d_{n+1}(s_1,t_2)+d_{n+1}(s_2,t_1)+d_{n+1}(s_2,t_2)}{8} \right\rceil
    \end{equation}
\end{theorem}
\begin{proof}
    Let us assume that the shortest $\{s_2, t_1\}$ path in $S^{n+1}_3$ does not pass through $s_1$ or $t_2$. We know that in a path, between any two vertices in $S^{n+1}_3$, clique and non-clique edges occur alternately.
    The shortest $\{s_2, t_1\}$ path not passing through $s_1$ and $t_2$ in $S^{n+1}_3$ starts with a clique edge and ends with a clique edge. The shortest $\{s, t\}$ in $\widehat{S}^n_3$ has the same number of edges as the number of clique edges in the shortest $\{s_2, t_1\}$ path in $S^{n+1}_3$, that is
    \begin{equation}\label{ref:eq2}
    \widehat{d}_n(s,t)=\frac{d_{n+1}(s_2,t_1)+1}{2}
    \end{equation}
    \\
    For the rest of the paths, there are 4 exhaustive cases:
    \begin{enumerate}
        \item[\mylabel{itm:1}{Case 1}:] The $\{s_1, t_1\}$ path passes through $s_2$, the $\{s_2, t_2\}$ path passes through $t_1$, and the $\{s_1, t_2\}$ path passes through both $s_2$ and $t_1$. (For example, consider the case $s_1=002$, $s_2=020$, $t_1=022$, and $t_2=200$ in Figure \ref{fig:S331}, where the shortest path for any pair of vertices follow the red line.)\\
        Therefore, we get:
        \begin{equation}\label{ref:eq3}
        \widehat{d}_n(s,t)=\frac{d_{n+1}(s_1,t_1)}{2}
        \end{equation}
        \begin{equation}\label{ref:eq4}
        \widehat{d}_n(s,t)=\frac{d_{n+1}(s_2,t_2)}{2}
        \end{equation}
        \begin{equation}\label{ref:eq5}
        \widehat{d}_n(s,t)=\frac{d_{n+1}(s_1,t_2)-1}{2}
        \end{equation}
        \\
        Combining equations \ref{ref:eq3}, \ref{ref:eq4}, and \ref{ref:eq5}, we get 
        \begin{equation}\label{ref:eq6}
        \widehat{d}_n(s,t)=\frac{1}{8}[d_{n+1}(s_1,t_1)+d_{n+1}(s_1,t_2)+d_{n+1}(s_2,t_1)+d_{n+1}(s_2,t_2)]
        \end{equation}

        \tikzset{every picture/.style={line width=0.75pt}} 
    \begin{figure}
    \centering
    \begin{tikzpicture}[x=1pt,y=1pt,yscale=-1,xscale=1]
    
    \draw    (331,80.41) -- (405,209.41) ;
    \node at (414,209.41) {$222$}; 
    \draw    (331,80.41) -- (257,209.41) ;
    \node at (331,73.41) {$000$}; 
    \draw    (405,209.41) -- (257,209.41) ;
    \node at (248,209.41) {$111$};
    \draw    (337,209.41) -- (372.24,151.5) ;
    \draw    (325,209.41) -- (289.76,151.5) ;
    \draw    (295,142) -- (367,142) ;
    \node at (340,216.41) {$211$};
    \node at (322,216.41) {$122$};
    \node at (388.24,151.5) {$200 (t_2)$};
    \node at (280.76,151.5) {$100$};
    \node at (286,140) {$011$};
    \node at (383,140) {$022 (t_1)$};
    \draw    (326,142) -- (310,117) ;
    \draw    (336,142) -- (352,117) ;
    \draw    (314,109) -- (348,109) ;
    \node at (323,148) {$012$};
    \node at (301,117) {$010$};
    \node at (339,148) {$021$};
    \node at (369,117) {$020 (s_2)$};
    \node at (305,109) {$001$};
    \node at (365,109) {$002 (s_1)$};
    \draw    (376,209.41) -- (390.5,185.41) ;
    \draw    (366,209.41) -- (351.5,185.41) ;
    \draw    (356,177.41) -- (387,177.41) ;
    \node at (379,216.41) {$212$};
    \node at (400,185.41) {$220$};
    \node at (363,216.41) {$221$};
    \node at (342,185.41) {$210$};
    \node at (347,177.41) {$201$};
    \node at (396,177.41) {$202$};
    \draw    (296,209.41) -- (310.5,185.41) ;
    \draw    (286,209.41) -- (270.5,185.41) ;
    \draw    (275,177.41) -- (306,177.41) ;
    \node at (299,216.41) {$121$};
    \node at (319.5,185.41) {$102$};
    \node at (283,216.41) {$112$};
    \node at (261.5,185.41) {$101$};
    \node at (266,177.41) {$110$};
    \node at (315,177.41) {$120$};
    \draw[red]   (348,109) -- (372.24,151.5) ;
    \end{tikzpicture}
    \caption{Sierpi\'{n}ski graph $S^3_3$}
    \label{fig:S331}
    \end{figure}
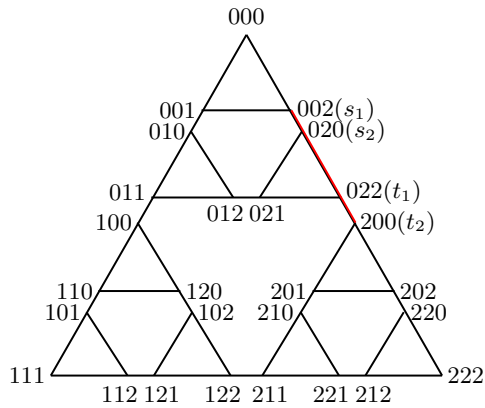
            
        \item[\mylabel{itm:2}{Case 2}:] The $\{s_1, t_2\}$ path does not pass through $s_2$ or $t_1$, the $\{s_2, t_2\}$ path passes through $t_1$ (or $s_1$), and the $\{s_1, t_1\}$ path passes through $t_2$ (or $s_2$). (For example, consider the case $s_1=012$, $s_2=021$, $t_1=211$, and $t_2=122$ in Figure \ref{fig:S332}, where the shortest path for each pair of vertices follows the red line.)
        
        Therefore, we get:
        \begin{equation}\label{ref:eq7}
        \widehat{d}_n(s,t)=\frac{d_{n+1}(s_1,t_1)}{2}
        \end{equation}
        \begin{equation}\label{ref:eq8}
        \widehat{d}_n(s,t)=\frac{d_{n+1}(s_2,t_2)}{2}
        \end{equation}
        \begin{equation}\label{ref:eq9}
        \widehat{d}_n(s,t)=\frac{d_{n+1}(s_1,t_2)+1}{2}
        \end{equation}
        \\
        Combining equations \ref{ref:eq7}, \ref{ref:eq8}, and \ref{ref:eq9}, we get 
        \begin{align*}
        \widehat{d}_n(s,t)=\frac{1}{8}[d_{n+1}(s_1,t_1)+d_{n+1}(s_1,t_2)+d_{n+1}(s_2,t_1)+d_{n+1}(s_2,t_2)+2]
        \\    
        \\
        =\left\lceil \frac{d_{n+1}(s_1,t_1)+d_{n+1}(s_1,t_2)+d_{n+1}(s_2,t_1)+d_{n+1}(s_2,t_2)}{8} \right\rceil
        \end{align*}

        \tikzset{every picture/.style={line width=0.75pt}} 
    \begin{figure}
    \centering
    \begin{tikzpicture}[x=1pt,y=1pt,yscale=-1,xscale=1]
    
    \draw    (331,80.41) -- (405,209.41) ;
    \node at (414,209.41) {$222$}; 
    \draw    (331,80.41) -- (257,209.41) ;
    \node at (331,73.41) {$000$}; 
    \draw    (405,209.41) -- (257,209.41) ;
    \node at (248,209.41) {$111$};
    \draw    (337,209.41) -- (372.24,151.5) ;
    \draw    (325,209.41) -- (289.76,151.5) ;
    \draw    (295,142) -- (367,142) ;
    \node at (340,216.41) {$211$};
    \node at (340,224.41) {$(t_1)$};
    \node at (322,216.41) {$122$};
    \node at (322,224.41) {$(t_2)$};
    \node at (381.24,151.5) {$200$};
    \node at (280.76,151.5) {$100$};
    \node at (286,140) {$011$};
    \node at (376,140) {$022$};
    \draw    (326,142) -- (310,117) ;
    \draw    (336,142) -- (352,117) ;
    \draw    (314,109) -- (348,109) ;
    \node at (323,148) {$012$};
    \node at (323,156) {$(s_1)$};
    \node at (301,117) {$010$};
    \node at (339,148) {$021$};
    \node at (339,156) {$(s_2)$};
    \node at (361,117) {$020$};
    \node at (305,109) {$001$};
    \node at (357,109) {$002$};
    \draw    (376,209.41) -- (390.5,185.41) ;
    \draw    (366,209.41) -- (351.5,185.41) ;
    \draw    (356,177.41) -- (387,177.41) ;
    \node at (379,216.41) {$212$};
    \node at (400,185.41) {$220$};
    \node at (363,216.41) {$221$};
    \node at (342,185.41) {$210$};
    \node at (347,177.41) {$201$};
    \node at (396,177.41) {$202$};
    \draw    (296,209.41) -- (310.5,185.41) ;
    \draw    (286,209.41) -- (270.5,185.41) ;
    \draw    (275,177.41) -- (306,177.41) ;
    \node at (299,216.41) {$121$};
    \node at (319.5,185.41) {$102$};
    \node at (283,216.41) {$112$};
    \node at (261.5,185.41) {$101$};
    \node at (266,177.41) {$110$};
    \node at (315,177.41) {$120$};
    \draw[red]   (372.24,151.5) -- (337,209.41) ;
    \draw[red]   (325, 209.41) -- (337,209.41) ;
    \draw[red]   (289.76,151.5) -- (325,209.41) ;
    \draw[red]   (289.76,151.5) -- (295,142) ;
    \draw[red]   (372.24,151.5) -- (367,142) ;
    \draw[red]   (336,142) -- (367,142) ;
    \draw[red]   (326,142) -- (295,142) ;
    \end{tikzpicture}
    \caption{Sierpi\'{n}ski graph $S^3_3$}
    \label{fig:S332}
    \end{figure}
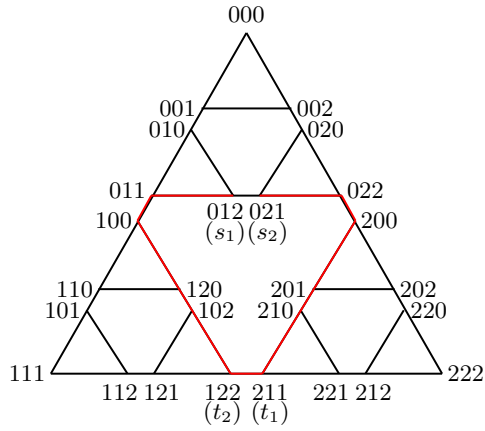
    
        \item[\mylabel{itm:3}{Case 3}:] Here, we have $2$ subcases, which are similar but different in the way they are represented. 
        \begin{enumerate}
        \item[(3a)] The $\{s_2, t_2\}$ path does not pass through $s_1$ or $t_1$, the $\{s_1, t_1\}$ path passes through $s_2$, and the $\{s_1, t_2\}$ path passes through $s_2$. (For example, consider the case $s_1=002$, $s_2=020$, $t_1=212$, and $t_2=221$ in Figure \ref{fig:S333}, where the shortest path between any pair of vertices follows a path along the red line.)\\
        Therefore, we get:
        \begin{equation}\label{ref:eq10}
        \widehat{d}_n(s,t)=\frac{d_{n+1}(s_1,t_1)}{2}
        \end{equation}
        \begin{equation}\label{ref:eq11}
        \widehat{d}_n(s,t)=\frac{d_{n+1}(s_2,t_2)+1}{2}
        \end{equation}
        \begin{equation}\label{ref:eq12}
        \widehat{d}_n(s,t)=\frac{d_{n+1}(s_1,t_2)}{2}
        \end{equation}
        
        Combining equations \ref{ref:eq10}, \ref{ref:eq11}, and \ref{ref:eq12}, we get 
        \begin{align*}
        \widehat{d}_n(s,t)=\frac{1}{8}[d_{n+1}(s_1,t_1)+d_{n+1}(s_1,t_2)+d_{n+1}(s_2,t_1)+d_{n+1}(s_2,t_2)+2]
        \\
        \\
        =\left\lceil \frac{d_{n+1}(s_1,t_1)+d_{n+1}(s_1,t_2)+d_{n+1}(s_2,t_1)+d_{n+1}(s_2,t_2)}{8} \right\rceil
        \end{align*}

        \tikzset{every picture/.style={line width=0.75pt}} 
    \begin{figure}
    \centering
    \begin{tikzpicture}[x=1pt,y=1pt,yscale=-1,xscale=1]
    
    \draw    (331,80.41) -- (405,209.41) ;
    \node at (414,209.41) {$222$}; 
    \draw    (331,80.41) -- (257,209.41) ;
    \node at (331,73.41) {$000$}; 
    \draw    (405,209.41) -- (257,209.41) ;
    \node at (248,209.41) {$111$};
    \draw    (337,209.41) -- (372.24,151.5) ;
    \draw    (325,209.41) -- (289.76,151.5) ;
    \draw    (295,142) -- (367,142) ;
    \node at (340,216.41) {$211$};
    \node at (322,216.41) {$122$};
    \node at (381.24,151.5) {$200$};
    \node at (280.76,151.5) {$100$};
    \node at (286,140) {$011$};
    \node at (376,140) {$022$};
    \draw    (326,142) -- (310,117) ;
    \draw    (336,142) -- (352,117) ;
    \draw    (314,109) -- (348,109) ;
    \node at (323,148) {$012$};
    \node at (301,117) {$010$};
    \node at (339,148) {$021$};
    \node at (369,117) {$020 (s_2)$};
    \node at (305,109) {$001$};
    \node at (365,109) {$002 (s_1)$};
    \draw    (376,209.41) -- (390.5,185.41) ;
    \draw    (366,209.41) -- (351.5,185.41) ;
    \draw    (356,177.41) -- (387,177.41) ;
    \node at (379,216.41) {$212$};
    \node at (379,224.41) {$(t_1)$};
    \node at (400,185.41) {$220$};
    \node at (363,216.41) {$221$};
    \node at (363,224.41) {$(t_2)$};
    \node at (342,185.41) {$210$};
    \node at (347,177.41) {$201$};
    \node at (396,177.41) {$202$};
    \draw    (296,209.41) -- (310.5,185.41) ;
    \draw    (286,209.41) -- (270.5,185.41) ;
    \draw    (275,177.41) -- (306,177.41) ;
    \node at (299,216.41) {$121$};
    \node at (319.5,185.41) {$102$};
    \node at (283,216.41) {$112$};
    \node at (261.5,185.41) {$101$};
    \node at (266,177.41) {$110$};
    \node at (315,177.41) {$120$};
    \draw[red]    (348,109) -- (372.24,151.5);
    \draw[red]    (390.5,185.41) -- (376,209.41);
    \draw[red]    (351.5,185.41) -- (366,209.41);
    \draw[red]    (390.5,185.41) -- (372.24,151.5);
    \draw[red]    (351.5,185.41) -- (372.24,151.5);
    \end{tikzpicture}
    \caption{Sierpi\'{n}ski graph $S^3_3$}
    \label{fig:S333}
    \end{figure}
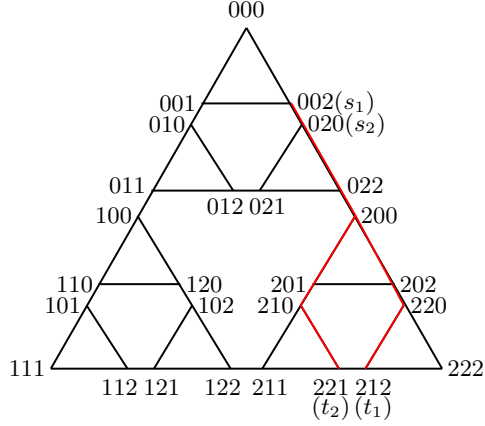

        \item[(3b)] The $\{s_1, t_1\}$ path does not pass through $s_2$ or $t_2$, the $\{s_2, t_2\}$ path passes through $t_1$, and the $\{s_1, t_2\}$ path passes through $t_1$. (For example, consider the case $s_1=002$, $s_2=020$, $t_1=212$, and $t_2=221$ in Figure \ref{fig:S333}, where the shortest path between any pair of vertices follows a path along the red line.)\\
        Therefore, we get:
        \begin{equation}\label{ref:eq13}
        \widehat{d}_n(s,t)=\frac{d_{n+1}(s_1,t_1)+1}{2}
        \end{equation}
        \begin{equation}\label{ref:eq14}
        \widehat{d}_n(s,t)=\frac{d_{n+1}(s_2,t_2)}{2}
        \end{equation}
        \begin{equation}\label{ref:eq15}
        \widehat{d}_n(s,t)=\frac{d_{n+1}(s_1,t_2)}{2}
        \end{equation}
        
        Combining equations \ref{ref:eq13}, \ref{ref:eq14}, and \ref{ref:eq15}, we get 
        \begin{align*}
        \widehat{d}_n(s,t)=\frac{1}{8}[d_{n+1}(s_1,t_1)+d_{n+1}(s_1,t_2)+d_{n+1}(s_2,t_1)+d_{n+1}(s_2,t_2)+2]
        \\
        \\
        =\left\lceil \frac{d_{n+1}(s_1,t_1)+d_{n+1}(s_1,t_2)+d_{n+1}(s_2,t_1)+d_{n+1}(s_2,t_2)}{8} \right\rceil
        \end{align*}
        \end{enumerate}
        From here on, we will be referring to any pair of vertices $(s_1,s_2)$ and $(t_1,t_2)$ adhering to cases 3a or 3b as \ref{itm:3}.
    
        \item[\mylabel{itm:4}{Case 4}:] The paths $\{s_2, t_2\}$, $\{s_1, t_1\}$, and $\{s_1, t_2\}$ do not encounter the remaining 2 vertices. (For example, consider the case $s_1=002$, $s_2=020$, $t_1=112$, and $t_2=121$ in Figure \ref{fig:S33})
        Therefore, we get:
        \begin{equation}\label{ref:eq16}
        \widehat{d}_n(s,t)=\frac{d_{n+1}(s_1,t_1)+1}{2}
        \end{equation}
        \begin{equation}\label{ref:eq17}
        \widehat{d}_n(s,t)=\frac{d_{n+1}(s_2,t_2)+1}{2}
        \end{equation}
        \begin{equation}\label{ref:eq18}
        \widehat{d}_n(s,t)=\frac{d_{n+1}(s_1,t_2)+1}{2}
        \end{equation}
        
        Combining equations \ref{ref:eq16}, \ref{ref:eq17}, and \ref{ref:eq18}, we get 
        \begin{align*}
        \widehat{d}_n(s,t)=\frac{1}{8}[d_{n+1}(s_1,t_1)+d_{n+1}(s_1,t_2)+d_{n+1}(s_2,t_1)+d_{n+1}(s_2,t_2)+4]
    \\
    \\
        =\left\lceil \frac{d_{n+1}(s_1,t_1)+d_{n+1}(s_1,t_2)+d_{n+1}(s_2,t_1)+d_{n+1}(s_2,t_2)}{8} \right\rceil
        \end{align*}
    \end{enumerate}
    Thus, from these four exhaustive cases, we can safely conclude that equation \ref{ref:eq1} holds:
    \begin{align*}
    \widehat{d}_n(s,t)=\left\lceil \frac{d_{n+1}(s_1,t_1)+d_{n+1}(s_1,t_2)+d_{n+1}(s_2,t_1)+d_{n+1}(s_2,t_2)}{8} \right\rceil
    \end{align*}
    \tikzset{every picture/.style={line width=0.75pt}} 
    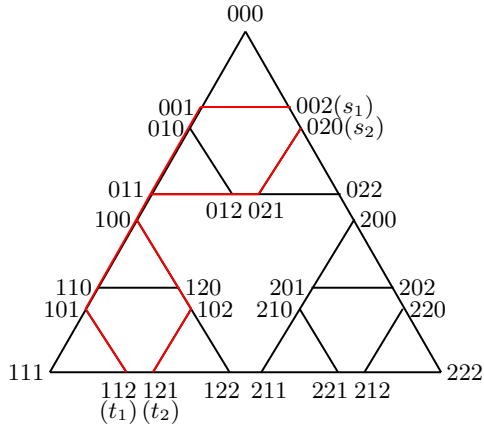
\begin{figure}
    \centering
    \begin{tikzpicture}[x=1pt,y=1pt,yscale=-1,xscale=1]
    
    \draw    (331,80.41) -- (405,209.41) ;
    \node at (414,209.41) {$222$}; 
    \draw    (331,80.41) -- (257,209.41) ;
    \node at (331,73.41) {$000$}; 
    \draw    (405,209.41) -- (257,209.41) ;
    \node at (248,209.41) {$111$};
    \draw    (337,209.41) -- (372.24,151.5) ;
    \draw    (325,209.41) -- (289.76,151.5) ;
    \draw    (295,142) -- (367,142) ;
    \node at (340,216.41) {$211$};
    \node at (322,216.41) {$122$};
    \node at (381.24,151.5) {$200$};
    \node at (280.76,151.5) {$100$};
    \node at (286,140) {$011$};
    \node at (376,140) {$022$};
    \draw    (326,142) -- (310,117) ;
    \draw    (336,142) -- (352,117) ;
    \draw    (314,109) -- (348,109) ;
    \node at (323,148) {$012$};
    \node at (301,117) {$010$};
    \node at (339,148) {$021$};
    \node at (369,117) {$020 (s_2)$};
    \node at (305,109) {$001$};
    \node at (365,109) {$002 (s_1)$};
    \draw    (376,209.41) -- (390.5,185.41) ;
    \draw    (366,209.41) -- (351.5,185.41) ;
    \draw    (356,177.41) -- (387,177.41) ;
    \node at (379,216.41) {$212$};
    \node at (400,185.41) {$220$};
    \node at (363,216.41) {$221$};
    \node at (342,185.41) {$210$};
    \node at (347,177.41) {$201$};
    \node at (396,177.41) {$202$};
    \draw    (296,209.41) -- (310.5,185.41) ;
    \draw    (286,209.41) -- (270.5,185.41) ;
    \draw    (275,177.41) -- (306,177.41) ;
    \node at (299,216.41) {$121$};
    \node at (299,224.41) {$(t_2)$};
    \node at (319.5,185.41) {$102$};
    \node at (283,216.41) {$112$};
    \node at (283,224.41) {$(t_1)$};
    \node at (261.5,185.41) {$101$};
    \node at (266,177.41) {$110$};
    \node at (315,177.41) {$120$};
    \draw[red]    (348,109) -- (314,109);
    \draw[red]    (352,117) -- (336,142);
    \draw[red]    (336,142) -- (295,142);
    \draw[red]    (314,109) -- (295,142);
    \draw[red]    (289.76,151.5) -- (295,142);
    \draw[red]    (270.5,185.41) -- (289.76,151.5);
    \draw[red]    (310.5,185.41) -- (289.76,151.5);
    \draw[red]    (270.5,185.41) -- (286,209.41);
    \draw[red]    (310.5,185.41) -- (296,209.41);
    \end{tikzpicture}
    \caption{Sierpi\'{n}ski graph $S^3_3$}
    \label{fig:S334}
    \end{figure}
    
\end{proof}
\textbf{Definitions:}
\begin{enumerate}
    \item $\delta^{(n)}(s,t)=d_{n+1}(s_1,t_1)+d_{n+1}(s_1,t_2)+d_{n+1}(s_2,t_1)+d_{n+1}(s_2,t_2)$, \\for $s,t \in V(\widehat{S}^{n}_3) \setminus \{\widehat{0}, \widehat{1}, \widehat{2}\}$
    \item $\delta^{(n)}(s)=\sum_{t \in V'(\widehat{S}^{n}_3)}\delta^{(n)}(s,t)$, where $V'(\widehat{S}^{n}_3)=V(\widehat{S}^{n}_3)\setminus\{\widehat{0}, \widehat{1}, \widehat{2}\}$
\end{enumerate}
When now consider the distances between any two vertices $s$ and $t$, where $s, t \in V(\widehat{S}^n)\setminus\{\widehat{0}, \widehat{1}, \widehat{2}\}$, and try to find bounds for $\delta^{(n)}(s,t)$. The following lemma would be crucial in determining the bounds for $\delta^{(n)}(s,t)$.
\begin{lem}\label{obs:delta}
    $\delta^{(n)}(s,t) \equiv x$ $(\mathrm{mod}$ $8)$, where $x \in \{0, 4, 6\}$
\end{lem}
\begin{remark}\label{thm:r6}
If we consider some $s \in V(\widehat{S}^n)\setminus\{\widehat{0}, \widehat{1}, \widehat{2}\}$, then:
$$d'_{n+1}(s_1)+d'_{n+1}(s_2)=\delta^{(n)}(s).$$
\end{remark}
We now look at the distances of all non-median vertices from all other \\ vertices of the Sierpi\'{n}ski triangle graph $\widehat{S}^n_3$.
\begin{lem}~\label{obs-As}
    If $V'(\widehat{S}^n_3)=V(\widehat{S}^n_3)\setminus \{\widehat{i}, \widehat{j}, \widehat{k}\}$, then for any $s \in V'(\widehat{S}^n_3) \setminus M(\widehat{S}^n_3)$, we have 
    \begin{equation*}
    \delta^{(n)}(s)+3^n-3 \le 8\widehat{d}'_{n}(s)
    \end{equation*}
\end{lem}
\begin{proof}
    We start by defining non-clique edges for Sierpi\'nski graphs.
    \begin{definition}
        Non-clique edges in a Sierpi\'nski graphs are those edges which do not belong to a $3$-cycle and which when contracted from the vertices of $\widehat{S}^n_3$.
    \end{definition}
    We consider here all non-clique edges in subgraph $iS^n_3$, whose end vertices are non-median vertices. We extend this result to the rest of the graph $S^{n+1}_3$ by symmetry. We add some value $r \in \{0,1,2,3,4,5,6,7\}$ to some $x$, so that $x$ is perfectly divisible by 8, that is, $\lceil \frac{x}{8} \rceil$=$\frac{x+r}{8}$. Before we delve into the proof, we define the concept of parallel edges in Sierpi\'nski graphs.
    \begin{definition}
        In a Sierpi\'nski graph $S^n_3$, any two edges $\{s_1,s_2\}$ and $\{t_1,t_2\}$ are called parallel edges if $d(s_1,t_1)=d(s_2,t_2)$, given that $s_1=\underline{s'_1}j$, $s_2=\underline{s'_2}k$, $t_1=\underline{t'_1}j$, and $t_2=\underline{t'_2}k$, where $\underline{s'_1}, \underline{s'_2}, \underline{t'_1}, \underline{t'_2} \in T^{n-1}$ and $j, k \in T$.
    \end{definition}
    \begin{enumerate}
        \item[\mylabel{itm:A1}{Case A1}:] We consider all non-clique edges $\{s_1, s_2\} \in E(S^{n+1}_3)$ parallel to the edge $\{ik^n, ki^n\}$. All non-clique edges $\{t_1, t_2\}$ in the subgraph $jS^n_3$ have value $\delta^{(n)}(s,t) \equiv 4$ or $6$ $(\mathrm{mod}$ $8)$ as they follow a path similar to \ref{itm:2} or \ref{itm:4}. (For example, let us consider $s_1=0022$, $s_2=0200$, $t_1=1022$ and $t_2=1200$ in Figure~\ref{fig:S43}) There are $\frac{3^n-3}{2}$ such edges.
        \item[\mylabel{itm:A2}{Case A2}:] We consider all non-clique edges $\{s_1, s_2\} \in E(S^{n+1}_3)$ parallel to the edge $\{ij^n, ji^n\}$. All non-clique edges $\{t_1, t_2\}$ in the subgraph $kS^n_3$ have value $\delta^{(n)}(s,t) \equiv 4$ or $6$ $(\mathrm{mod}$ $8)$ as they follow a path similar to \ref{itm:2} or \ref{itm:4}. (For example, let us consider $s_1=0001$, $s_2=0010$, $t_1=2012$ and $t_2=2021$ in Figure~\ref{fig:S43}) There are $\frac{3^n-3}{2}$ such edges.
        \item[\mylabel{itm:A3}{Case A3}:] We now consider all edges in $S^{n+1}_3$ of form $\{s_1, s_2\}$=$\{i^mjk^{n-m}, i^mkj^{n-m}\}$, for all $2 \le m \le n-1$. (For example, let us consider $s_1=0012$, and $s_2=0021$ in Figure~\ref{fig:S43})
        \begin{enumerate}
        \item[\mylabel{itm:A3a}{Case A3a}:] All non-clique edges $\{t_1, t_2\}$ parallel to $\{s_1, s_2\}$ in the subgraphs of form $i^ljS^{n-l}_3$ and $i^lkS^{n-l}_3$ (for all $0 \le l \le m-1$) have value $\delta^{(n)}(s,t) \equiv 6$ $(\mathrm{mod}$ $8)$ as they follow a path similar to \ref{itm:3}. (For example, let us consider $t_1=0212$ and $t_2=0221$ in Figure~\ref{fig:S43}) For each of these, there are $2(3^{n-l-2}+3^{n-l-3}+\cdots+3^0)=$ such edges. 
        \item[\mylabel{itm:A3b}{Case A3b}:] Now we consider the edges $e_1=\{ij^n, ji^n\}$ and $e_2=\{ik^n, ki^n\}$, where $e_1, e_2 \in E(S^{n+1}_3)$. All non-clique edges parallel to edges $e_1$ and $e_2$ in subgraphs $i^mjS^{n-m}_3$ and $i^mkS^{n-m}_3$, respectively have value $\delta^{(n)}(s,t) \equiv 6$ $(\mathrm{mod}$ $8)$ as they follow a path similar to \ref{itm:3}. There are $2(3^{n-m-2}+3^{n-m-3}+\cdots+3^0)$ such edges. 
        \item[\mylabel{itm:A3c}{Case A3c}:] All non-clique edges parallel to edges $e_2$ and $e_1$ in all subgraphs of form $i^ljS^{n-l}_3$ and $i^lkS^{n-l}_3$ (for all $m+1 \le l \le n-2$), respectively have value $\delta^{(n)}(s,t) \equiv 6$ $(\mathrm{mod}$ $8)$ as they follow a path similar to \ref{itm:3}. There are $2(3^{n-l-2}+3^{n-l-3}+\cdots+3^0)$ such edges.
        \item[\mylabel{itm:A3d}{Case A3d}:] Moreover, all non-clique edges $\{t_1, t_2\}$ of form $\{i^rjk^{n-r}, i^rkj^{n-r}\}$ \\ (for all $0 \le r \le n-1$ and $r \neq m$) have value $\delta^{(n)}(s,t) \equiv 6$ $(\mathrm{mod}$ $8)$  as it follows a path similar to \ref{itm:2}. (For example, let us consider $t_1=0122$ and $t_2=0211$ in Figure~\ref{fig:S43}) There are $n-1$ such edges.
        \end{enumerate}
        $\therefore$ Total no.of edges with value $\delta^{(n)}(m,t) \equiv 6$ $(\mathrm{mod}$ $8)$ in \ref{itm:A3} are:
        \begin{align*}
        \sum_{m=0}^{n-2}(2\sum_{l=0}^{l=m} 3^l)+(n-1)&=\sum_{m=0}^{n-2}(3^{m+1}-1)+(n-1)
    \\
        &=\sum_{m=0}^{n-2}(3^{m+1})
    \\ 
        &=\frac{3^n-3}{2}
        \end{align*}
        \item[\mylabel{itm:A4}{Case A4}:] Now we consider all the remaining non-clique edges $\{s_1, s_2\}$. All these non-clique edges are parallel to the edge $\{jk^n, kj^n\}$ and lie in some subgraph of form $i^mjS^{n-m}_3$ or $i^mkS^{n-m}_3$, where $m \in [1, n-2]$. (For example, let us consider $s_1=0212$ and $s_2=0221$ in Figure~\ref{fig:S43}) 
        \begin{enumerate}
        \item[\mylabel{itm:A4a}{Case A4a}:] All non-clique edges $\{t_1,t_2\}$ parallel to $\{s_1,s_2\}$ in all subgraphs of form $i^ljS^{n-l}$ and $i^lkS^{n-l}$, where $l \in [0, m-1]$, have value $\delta^{(n)}(s,t) \equiv 6$ $(\mathrm{mod}$ $8)$. (For example, let us consider $t_1=2012$ and $t_2=2021$ in Figure~\ref{fig:S43}) For each of these cases, we have $2\frac{3^{n-l}-3}{2\times3}$ such edges. 
        \item[\mylabel{itm:A4b}{Case A4b}:] Now we consider the edges $e_1=\{i^mki^{n-m}, i^{m+1}k^{n-m}\}$ and $e_2=\{i^mji^{n-m}, i^{m+1}j^{n-m}\}$, where $e_1, e_2 \in E(S^{n+1}_3)$. If $\{s_1, s_2\} \in E(i^mjS^{n-m}_3)$ then we consider edge $e_2$ and if $\{s_1, s_2\} \in E(i^mkS^{n-m}_3)$ then we consider edge $e_2$. We name our choice of $j$ or $k$ as $p$ and our choice of edge $e_1$ or $e_2$ as $e$.
        
        For all non-clique edges $\{t_1, t_2\}$ parallel to $e$ in subgraph $i^m(3-i-p)S^{n-m}_3$ have value $\delta^{(n)}(s,t) \equiv 6$ $(\mathrm{mod}$ $8)$. (For example, let us consider $t_1=0101$ and $t_2=0110$ in Figure~\ref{fig:S43}) There are $\frac{3^{n-m}-3}{2\times3}$ such edges. For all non-clique edges $\{t_1, t_2\}$ parallel to $e$ in subgraph $i^{m+1}S^{n-m}_3$ have value $\delta^{(n)}(s,t) \equiv 4$ $(\mathrm{mod}$ $8)$. (For example, let us consider $t_1=0001$ and $t_2=0010$ in Figure~\ref{fig:S43}) There are $\frac{3^{n-m}-3}{2\times3}$ such edges. 
        \item[\mylabel{itm:A4c}{Case A4c}:] Moving on, now we consider the graph $\widehat{S}^n_3$. We consider the vertex $s$ in $\widehat{S}^n_3$ formed by contracting the edge $\{s_1, s_2\}$. Both $s_1$ and $s_2$ can be expressed as $i^mqs'_1$ and $i^mqs'_2$, where $s'_1, s'_2 \in T^{n-m}$ and $q \in {j, k}$. Next we denote $\widehat{d}_n(s, i^{m-1}q)$ by $d$. Now we consider all subgraphs of $\widehat{S}^n_3$ that have been formed by contracting non-clique edges of subgraphs of $S^{n+1}_3$ of form $i^l(3-i-q)S^{n-l}_3$, where $l \in [0, m-1]$.
        \begin{equation}
        \widehat{d}_n(s, i^lq)=2^{n-m}-d+\sum_{r=n-m}^{n-l+1}+2^{n-l}-x
        \label{eq:1}
        \end{equation}
        \begin{equation}
        \widehat{d}_n(s, i^{l+1})=d+\sum_{r=n-m}^{n-l+1}+2^{n-l}+x.
        \label{eq:2}
        \end{equation}
        Equating equations $\ref{eq:1}$ and $\ref{eq:2}$, we get $x=2^{n-m-1}-d$. So, we consider the vertex $t$ for which $t$ lies on the shortest path between $i^lq$ and $i^{l+1}$ and also, $\widehat{d}_n(t, i^{l+1})=x$. The edge $\{t_1, t_2\}$ in $S^{n+1}_3$ that is contracted to form vertex $t$ in $\widehat{S}^n_3$, also has value $\delta^{(n)}(s, t) \equiv 6$ $(\mathrm{mod}$ $8)$. (For example, let us consider $t_1=1220$ and $t_2=1202$ in Figure~\ref{fig:S43}) There are $m$ such edges in $S^{n+1}_3$.
        \end{enumerate}
        Now we use $F$ to represent the set of all the non-clique edges $t=\{t_1, t_2\}$ mentioned above in cases \ref{itm:A4a}, \ref{itm:A4b}, and \ref{itm:A4c}.
        \begin{align*}
        \sum_{t \in F}& \frac{\delta^{(n)}(s, t)+r}{8} \\
        &=\frac{1}{8}\left [\sum_{t \in F}\delta^{(n)}(s, t)+\sum_{r=0}^{m-1}\{\frac{3^{n-r}-3}{2\times3}(2\times2)\}+\frac{2(3^{n-m}-3)}{2\times3}+\frac{4(3^{n-m}-3)}{2\times3}+2 m \right ]
\\      
        &=\frac{1}{8}\left [\sum_{t \in F}\delta^{(n)}(s, t)+3^n-3 \right ]
\end{align*}
\begin{align*}
       \text{Therefore,} \sum_{t \in V'(\widehat{S}^n_3)} \frac{\delta^{(n)}(s, t)+r}{8} &\le \sum_{t \in V'(\widehat{S}^n_3)}\left\lceil \frac{\delta^{(n)}(s,t)}{8} \right\rceil
       \\
        \Longrightarrow \frac{\delta^{(n)}(s)+3^n-3}{8} &\le \sum_{t \in V'(\widehat{S}^n_3)}\left\lceil \frac{\delta^{(n)}(s,t)}{8} \right\rceil
        \\
        \Longrightarrow \delta^{(n)}(s)+3^n-3 &\le 8\widehat{d}'_{n}(s)
        \end{align*}
    \end{enumerate}
    However, for the first three cases in Lemma~\ref{obs-As}, namely \ref{itm:A1}, \ref{itm:A2}, and \ref{itm:A3}; let us consider E' to be the set of all aforementioned non-clique edges $t=\{t_1,t_2\}$, where $|E'|=\frac{3^n-3}{2}$
        \begin{align*}
        \sum_{t \in E'} \frac{\delta^{(n)}(s, t)+2}{8}&=\frac{1}{8}\left [\sum_{t \in E'} \delta^{(n)}(s, t)+3^n-3 \right ] \le \sum_{t \in E'} \frac{\delta^{(n)}(s, t)+r}{8}
        \\
        \therefore \sum_{t \in V'(\widehat{S}^n_3)} \frac{\delta^{(n)}(s, t)+r}{8} &\le \sum_{t \in V'(\widehat{S}^n_3)}\left\lceil \frac{\delta^{(n)}(s,t)}{8} \right\rceil
        \\
        \Longrightarrow \frac{1}{8}[\delta^{(n)}(s)+3^n-3] &\le \sum_{t \in V'(\widehat{S}^n_3)}\left\lceil \frac{\delta^{(n)}(s,t)}{8} \right\rceil
        \\
        \Longrightarrow \delta^{(n)}(s)+3^n-3 &\le 8\widehat{d}'_{n}(s)
        \end{align*}
\end{proof}
\begin{theorem}~\label{thm:r2}
$d'_n(s) > d'_{S^n}(m)$, where $m \in M(S^n)$ and $s \in V(S^n)\setminus M(S^n)$
\end{theorem}
\begin{proof}
        This follows directly from the fact that the average distance of the median vertices is less than the average distance of the non-median vertices.
\end{proof}
\begin{figure}[htbp]
     \centering
     \includegraphics[scale=0.3]{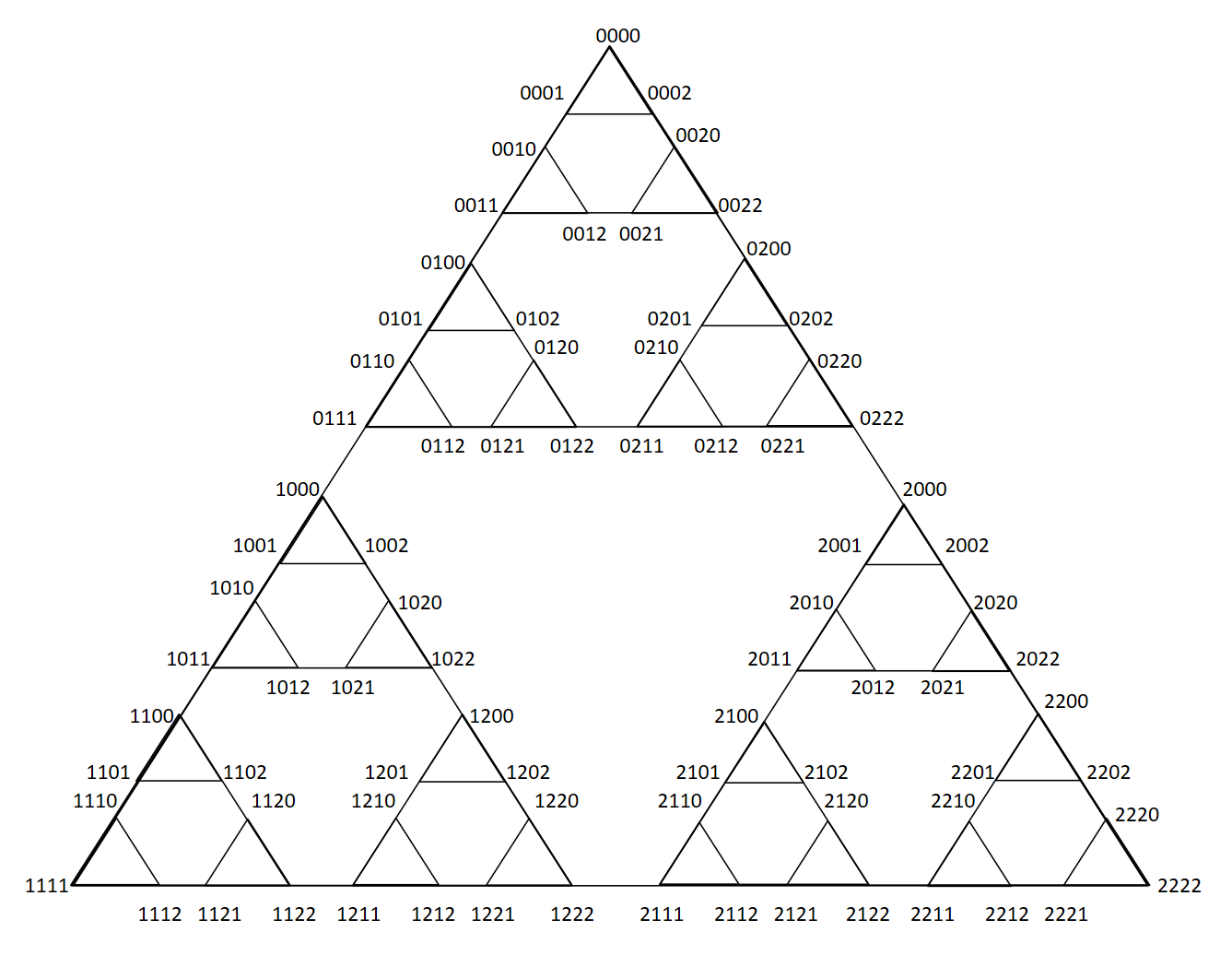}
     \caption{$S^4_3$}
     \label{fig:S43}
\end{figure}

We now try to evaluate about the distance of median vertices from other vertices in the Sierpi\'{n}ski triangle graph $\widehat{S}^n_3$.
\begin{lem}~\label{obs:Bs}
    If $V'(\widehat{S}^n_3)=V(\widehat{S}^n_3)\setminus \{\widehat{i}, \widehat{j}, \widehat{k}\}$, then for any $m \in M(\widehat{S}^n_3)$, we have exactly $\frac{3^n-3}{2}$ vertices $t \in V'(\widehat{S}^n_3)$ for which $\delta^{(n)}(m,t) \equiv 6 (\mathrm{mod} 8)$.
\end{lem}
\begin{proof}
    We will consider the general $S^{n+1}_3$ graph. We consider the edge $\{ijk^{n-1},ikj^{n-1}\}$ in the subgraph $iS^n_3$ and the edge $\{jk^n,kj^n\}$ and extend it to the rest of the graph $S^{n+1}_3$ by symmetry. For all $x$, we may add some value $r \in \{0,1,2,3,4,5,6,7\}$ to $x$ so that $x$ is perfectly divisible by $8$, that is, $\lceil \frac{x}{8} \rceil$=$\frac{x+r}{8}$.
    \begin{enumerate}
        \item[\mylabel{itm:B1}{Case B1}:] We consider the edge $\{m_1, m_2\}$=$\{ijk^{n-1},ikj^{n-1}\}$. (For example, let us consider $m_1=0122$ and $m_2=0211$ in Figure~\ref{fig:S43})
        \begin{enumerate}
        \item[\mylabel{itm:B1a}{Case B1a}:] All non-clique edges $\{t_1, t_2\}$ in subgraphs $jS^n_3$ and $kS^n_3$, that are parallel to the edge $\{m_1, m_2\}$ have value $\delta^{(n)}(m,t) \equiv 6$ $(\mathrm{mod}$ $8)$ as it follows a path similar to \ref{itm:3}. (For example, let us consider $t_1=1012$ and $t_2=1021$ in Figure~\ref{fig:S43}) There are $2(3^{n-2}+3^{n-3}+\cdots+3^0)$ such edges.
        \item[\mylabel{itm:B1b}{Case B1b}:] Now we consider the edges $e_1=\{i^2j^{n-1}, iji^{n-1}\}$ and $e_2=\{i^2k^{n-1}, iki^{n-1}\}$. Now all non-clique edges $\{t_1, t_2\}$ in subgraphs $ijS^{n-1}_3$ and $ikS^{n-1}_3$ parallel to edges $e_1$ and $e_2$, respectively have value $\delta^{(n)}(m,t) \equiv 6$ $(\mathrm{mod}$ $8)$  as it follows a path similar to \ref{itm:3}. (For example, let us consider $t_1=0202$ and $t_2=0220$ in Figure~\ref{fig:S43}) There are $2(3^{n-3}+3^{n-4}+\cdots+3^0)$ such edges. 
        \item[\mylabel{itm:B1c}{Case B1c}:] Moving on all non-clique edges $\{t_1, t_2\}$ in all subgraphs $i^mjS^{n-m}_3$ and $i^mkS^{n-m}_3$ (for all $2 \le m \le n-2$) parallel to $e_2$ and $e_1$, respectively have value $\delta^{(n)}(m,t) \equiv 6$ $(\mathrm{mod}$ $8)$ as it follows a path similar to \ref{itm:3}. For each of these cases, there are $2(3^{n-k-2}+3^{n-k-3}+\cdots+3^0)$ such edges. 
        \item[\mylabel{itm:B1d}{Case B1d}:] Moreover, all non-clique edges $\{t_1, t_2\}$ of form $\{i^mjk^{n-m}, i^mkj^{n-m}\}$ (for all $0 \le m \le n-1$ and $m \neq 1$) have value $\delta^{(n)}(m,t) \equiv 6$ $(\mathrm{mod}$ $8)$  as it follows a path similar to \ref{itm:2}. (For example, let us consider $t_1=0012$ and $t_2=0021$ in Figure~\ref{fig:S43}) There are $n-1$ such edges.
        \end{enumerate}
        $\therefore$ From cases \ref{itm:B1a}, \ref{itm:B1b}, \ref{itm:B1c}, and \ref{itm:B1d}, we get that total vertices $t \in V'(\widehat{S}^n_3)$ with value $\delta^{(n)}(m,t) \equiv 6$ $(\mathrm{mod}$ $8)$ are:
        \begin{align*}
        \sum_{m=0}^{n-2}(2\sum_{l=0}^{l=m} 3^l)+(n-1)&=\sum_{m=0}^{n-2}(3^{m+1}-1)+(n-1) 
        \\
        &=\sum_{m=0}^{n-2}(3^{m+1})
        \\
        &=\frac{3^n-3}{2}
        \end{align*}
        All other non-clique edges $\{t_1, t_2\}$ have value $\delta^{(n)}(m,t) \equiv 0$ $(\mathrm{mod}$ $8)$  as it follows a path similar to \ref{itm:1}.
        \item[\mylabel{itm:B2}{Case B2}:] Now we consider the edge $\{m_1, m_2\}$=$\{jk^n,kj^n\}$. (For example, let us \\ consider $m_1=1222$ and $m_2=2111$ in Figure~\ref{fig:S43})
        \begin{enumerate}
        \item[\mylabel{itm:B2a}{Case B2a}:] We consider the edges $e_1=\{ij^n, ji^n\}$ and $e_2=\{ik^n, ki^n\}$. All non-clique edges $\{t_1, t_2\}$ in subgraphs $jS^n_3$ and $kS^n_3$ parallel to edges $e_1$ and $e_2$, respectively have value $\delta^{(n)}(m,t) \equiv 6$ $(\mathrm{mod}$ $8)$  as it follows a path similar to \ref{itm:3}. (For example, let us consider $t_1=2102$ and $t_2=2120$ in Figure~\ref{fig:S43}) There are $2(3^{n-2}+3^{n-3}+\cdots+3^0)$ such edges. 
        \item[\mylabel{itm:B2b}{Case B2b}:] Next, all non-clique edges $\{t_1, t_2\}$ in all subgraphs $i^mjS^{n-m}_3$ and \\ $i^mkS^{n-m}_3$ (for all $1 \le m \le n-2$) parallel to $e_2$ and $e_1$, respectively have value $\delta^{(n)}(m,t) \equiv 6$ $(\mathrm{mod}$ $8)$ as it follows a path similar to \ref{itm:3}. (For example, let us consider $t_1=0201$ and $t_2=0210$ in Figure~\ref{fig:S43}) For each of these cases, there are $2(3^{n-k-2}+3^{n-k-3}+\cdots+3^0)$ such edges. 
        \item[\mylabel{itm:B2c}{Case B2c}:] Moreover, all non-clique edges $\{t_1, t_2\}$ of form $\{i^mjk^{n-m}, i^mkj^{n-m}\}$ (for all $1 \le m \le n-1$) have value $\delta^{(n)}(m,t) \equiv 6$ $(\mathrm{mod}$ $8)$  as it follows a path similar to \ref{itm:2}. (For example, let us consider $t_1=0122$ and $t_2=0211$ in Figure~\ref{fig:S43}) There are $n-1$ such edges.\\
        \end{enumerate}
        $\therefore$ From cases \ref{itm:B2a}, \ref{itm:B2b}, and \ref{itm:B2c}, we get that total vertices $t \in V'(\widehat{S}^n_3)$ with value $\delta^{(n)}(m,t) \equiv 6$ $(\mathrm{mod}$ $8)$ are:
        \begin{align*}
        \sum_{m=0}^{n-2}(2\sum_{l=0}^{l=m} 3^l)+(n-1)&=\sum_{m=0}^{n-2}(3^{m+1}-1)+(n-1)
        \\
        &=\sum_{m=0}^{n-2}(3^{m+1})
        \\
        &=\frac{3^n-3}{2}
        \end{align*}
        All other non-clique edges $\{t_1, t_2\}$ have value $\delta^{(n)}(m,t) \equiv 0$ $(\mathrm{mod}$ $8)$  as it follows a path similar to \ref{itm:1} as we will prove in Theorem \ref{thm:r3}.
    \end{enumerate}

    Thus, we get that for any $m \in M(\widehat{S}^n_3)$, we have exactly $\frac{3^n-3}{2}$ vertices $t \in V'(\widehat{S}^n_3)$ for which $\delta^{(n)}(m,t) \equiv 6 (\mathrm{mod} 8)$.
    
\end{proof}
\begin{theorem}\label{thm:r3}
    All non-clique edges other than the ones specified in proof of Lemma~\ref{obs:Bs} have value $\delta^{(n)}(m,t) \equiv 0$ $(\mathrm{mod}$ $8)$.
\end{theorem}
\begin{proof}
    We consider the non-clique edge in $\{m_1, m_2\}$=$\{i^ljk^{n-l}, i^lkj^{n-l}\}$, where $l \in \{0, 1\}$ and then we can extend this to the rest of the graph by virtue of symmetry. (For example, let us consider $s_1=0122$ and $s_2=0211$ in Figure~\ref{fig:S43}) 
    \begin{enumerate}
    \item[\mylabel{itm:C1}{Case C1}:] All non-clique edges in subgraphs $j^{p+1}S^{n-p}_3$ and $k^{p+1}S^{n-p}_3$, where $0 \le p \le l-1$, that are not parallel to the edge $\{jk^n, kj^n\}$ have value $\delta^{(n)}(m,t) \equiv 0$ $(\mathrm{mod}$ $8)$ as it follows a path similar to \ref{itm:1}. (For example, let us consider $t_1=2002$ and $t_2=2020$ in Figure~\ref{fig:S43})  
    \item[\mylabel{itm:C2}{Case C2}:] Now we consider the edges $e_1=\{ij^n, ji^n\}$ and edge $e_2=\{ik^n, ki^n\}$. All edges in subgraphs $i^ljS^{n-l}$ and $i^lkS^{n-l}$ not parallel to $e_1$ and $e_2$, respectively have value $\delta^{(n)}(m,t) \equiv 0$ $(\mathrm{mod}$ $8)$ as it follows a path similar to \ref{itm:1}. (For example, let us consider $t_1=0201$ and $t_2=0210$ in Figure~\ref{fig:S43}) 
    \item[\mylabel{itm:C3}{Case C3}:] Next, all non-clique edges not parallel to lines $e_2$ and $e_1$ in subgraphs $i^rjS^{n-r}$ and $i^rkS^{n-r}$, respectively where $r \in [l+1, n]$ have value $\delta^{(n)}(m,t) \equiv 0$ $(\mathrm{mod}$ $8)$ as it follows a path similar to \ref{itm:1}. 
    \item[\mylabel{itm:C4}{Case C4}:] All other non-clique edges of form either $\{i^mji^{n-m}, i^{m+1}j^{n-m}\}$ or \\ $\{i^mki^{n-m}, i^{m+1}k^{n-m}\}$, where $m \in [0, n-1]$, have value $\delta^{(n)}(m,t) \equiv 0$ $(\mathrm{mod}$ $8)$ as it follows a path similar to \ref{itm:1}. (For example, let us consider $t_1=0011$ and $t_2=0100$ in Figure~\ref{fig:S43}) 
    \end{enumerate}
All cases \ref{itm:C1}, \ref{itm:C2}, \ref{itm:C3}, and \ref{itm:C4}, are an exhaustive set of all vertices that are not accounted for in cases \ref{itm:B1}, and \ref{itm:B2}.

\end{proof}
\begin{theorem}\label{thm:r4}
    For $m_1, m_2 \in M(S^{n+1}_3)$, where $m_1$ and $m_2$ form a non-clique edge, we have $m \in V(\widehat{S}^n_3)$, which is formed by contracting the edge $\{m_1, m_2\}$, then we get that:
    \begin{equation*}
   8 \widehat{d}'_n(m) = \delta^{(n)}(m)+3^n-3
    \end{equation*}
\end{theorem}
\begin{proof}
    We know that, for all x, we may add some value $r \in \{0,1,2,3,4,5,6,7\}$ to $x$ so that $x$ is perfectly divisible by $8$, that is, $\lceil \frac{x}{8} \rceil$=$\frac{x+r}{8}$. As we have seen that, there are exactly $\frac{3^n-3}{2}$ non-clique edges that have value $\delta^{(n)}(m,t) \equiv 6$ $(\mathrm{mod}$ $8)$ for any choice of non-clique edge $\{m_1, m_2\}$, whose end vertices are medians of $S^{n+1}_3$. Let, us say the set of all the $\frac{3^n-3}{2}$ non-clique edges $\{t_1, t_2\}$ represented by $t$ is called $E'$.\\
    We use Theorem~\ref{thm:r3} and the already established equation in Lemma~\ref{obs:Bs}, we get that:\\
    \begin{align*}
    \sum_{t \in E'} \frac{\delta^{(n)}(m, t)+r}{8} &= \sum_{t \in E'} \frac{\delta^{(n)}(m, t)+2}{8}
    \\
    &=\frac{1}{8}\left [\sum_{t \in E'} \delta^{(n)}(m, t)+3^n-3 \right]
    \\
    \\
    \therefore \sum_{t \in V'(S^{n+1}_3)} \frac{\delta^{(n)}(m, t)+r}{8} &= \sum_{t \in V'(S^{n+1}_3)}\left \lceil \frac{\delta^{(n)}(m,t)}{8} \right \rceil
\\   
    \Longrightarrow \frac{\delta^{(n)}(m)+3^n-3}{8} &= \sum_{t \in V'(S^{n+1}_3)}\left \lceil \frac{\delta^{(n)}(m,t)}{8} \right \rceil 
    \\
    \Longrightarrow \delta^{(n)}(m)+3^n-3 &= 8 \widehat{d}'_n(m)
    \end{align*}
    
\end{proof}
\begin{theorem}\label{thm:r5}
    Let us consider all vertices $m \in V(\widehat{S}^n)$ that are formed by contracting non-clique edges of $S^{n+1}_3$, whose end vertices belong to $M(S^{n+1}_3)$. Then, $m \in M(\widehat{S}^n_3)$.
\end{theorem}

\begin{proof}
    Using Theorem~\ref{thm:r2} and Remark~\ref{thm:r6}, we get \textemdash
    \begin{align*}
    & d'_{n+1}(s_1)+d'_{n+1}(s_2)-2 > d'_{n+1}(m_1)+d'_{n+1}(m_2)-2 
\\
    & \Longrightarrow \delta^{(n)}(s) > \delta^{(n)}(m)
\\   
    & \text{ Now we use Lemma~\ref{obs-As} and Theorem~\ref{thm:r4}, we get \textemdash}\\
    & 8 \widehat{d}'_n(s) \ge \delta^{(n)}(s)+3^n-3 > \delta^{(n)}(m)+3^n-3 = 8 \widehat{d}'_n(m)
\\     
    & \Longrightarrow \widehat{d}'_n(s) > \widehat{d}'_n(m)    
    \end{align*}
    
\end{proof}

Thus, we have proved \textit{Theorem~\ref{hyp:basic}}.

\nopagebreak

\end{document}